\documentclass{article}

\usepackage{epstopdf}
\usepackage[bookmarksnumbered,colorlinks,bookmarks,citecolor=blue,linkcolor=blue,urlcolor=blue,breaklinks,linktocpage]{hyperref}
\usepackage[all]{hypcap}
\usepackage{subfigure}
\usepackage{amsmath, amssymb, amsthm, enumerate, graphicx, fullpage}

\theoremstyle{plain}
\newtheorem{theorem}{Theorem}[section]

\newtheorem{lemma}[theorem]{Lemma}
\newtheorem{proposition}[theorem]{Proposition}

\theoremstyle{definition}
\newtheorem{definition}[theorem]{Definition}

\theoremstyle{remark}

\begin{document}



\title{The axiom system of classical harmony}

\author{Andr\'{a}s T\'{o}bi\'{a}s\footnote{tobias\_AT\_ math.tu-berlin.de} \\ Institut f\"{u}r Mathematik, Technische Universit\"{a}t Berlin, Germany.\footnote{The work on which this paper is based was done between 2012--2014 for a national research competition and as my Bachelor's thesis \cite{szakdoga} at the Institute of Mathematics, Budapest University of Technology and Economics, Hungary, under the supervision of Dr.~\'{A}kos G.~Horv\'{a}th.}}


\maketitle

\begin{abstract}
 This paper provides a new mathematical axiom system for classical harmony, which is a prescriptive rule system for composing music, introduced in the second half of the 18th century. The clearest model of classical harmony is given by the homophonic four-part pieces of music. The form of these pieces is based on the earlier four-part chorale adaptations of J. S. Bach. Our paper logically structures the musical phenomena belonging to the research area of classical harmony. Its main result, the fundamental theorem of tonality, provides a way to construct a complete axiom system which incorporates the well-known classical compositional principles about chord changing and voice leading. In this axiom system, a piece complies with classical harmony if it satisfies the formal requirements of four-part homophony and it does not violate any classical chord-changing or modulational rules. \\
\textbf{Keywords:} axiom system, classical compositional principles, trichotomy of keys, homophonic four-part piece, fundamental theorem of tonality,  chord-changing rules
\end{abstract}






\section{Introduction}
The main goal of this paper is to provide a mathematical axiomatization for the strictly homophonic four-part model of classical harmony. This model is mainly based on J.~S.~Bach's four-part chorales, collected in \cite{bach}, which had been written before the start of Viennese classicism. Later, in the second half of the 18th century, almost unequivocal, prescriptive compositional principles were determined for this four-part model. By their nature, these compositional principles form a mathematical axiom system, as soon as all the basic notions of music corresponding to classical harmony are mathematically well-defined. In this paper, we present a possible way of logical ordering of these musical notions, and having these definitions we present a \emph{consistent and complete} axiom system which tells how to write homophonic four-part pieces.

Homophonic four-part pieces are interpreted as special right-continuous functions $M: \mathbb R^+ \to K^4$, where $K$ is an \emph{equal-tempered piano}. Here the real half-line refers to time. The special properties of $M$ are that for any $t \in \mathrm{Dom}~M$, $M(t)$ is a special kind of \emph{chord} in a special four-part version, and that the four voices always change their tones in the same time, yielding a \emph{chord change}. The chords in the ranges of the homophonic four-part pieces have to be some preliminarily given kinds of \emph{triads} and \emph{seventh chords} associated with a musical \emph{key}.

The formulation of the axiom system is the following. The axioms of classical harmony, which are called \emph{compositional principles}, determine whether a homophonic four-part piece whether complies with classical harmony or not. One of the first axioms, the so-called \emph{correctability condition} describes when a homophonic four-part piece complies with classical harmony \emph{apart from the chord-changing points}. Our \emph{fundamental theorem of tonality} (Theorem \ref{vécsey}) gives an equivalent condition for correctability of feasible pieces on a finite time interval. It claims that among these pieces, the correctable ones are exactly the ones which are \emph{locally tonal}. Its proof uses a sequential process for showing correctability of tonal pieces. Consequently, the chord-changing rules of classical harmony can be embedded in the axiom system as constraints. A correctable piece complies with classical harmony in a chord-changing point if none of these constraints is violated. Hence, until it does not cause inconsistence, one can add new chord-changing rules to the axiom system according to new results of music theory, or remove old ones. This way, exercises of writing correct four-part pieces consisting of given chords can virtually be interpreted as constraint programming problems. 

Some chord-changing compositional principles are mathematically described in \cite{pachetroy}. This paper investigates harmonizing four-part chorales algorithmically using constraints, and it turns out that it is not necessary to construct a full mathematical axiom system of classical harmony for this purpose. An axiom system is rather important for educational purposes: in order to teach musician students how to write \emph{perfect} examples for classical chord progression, one has to know all compositional principles. The task of students does not only include harmonizing given soprano melodies but also e.g. finding a correct voice configuration for a piece with given chords. Also in the case of more complex Classicist genres, generalized versions of the four-part rules apply. It is nevertheless true that although in Viennese classicism many compositional rules were indeed prescriptive for real music, composers' practice was substantially more irregular than what classical harmony would indicate, see e.g. \cite[p.~167]{ebcioglu}.

The axiomatization is completed by the \emph{modulation} (musical key change) rules, which give constraints on several consecutive chords which a modulation consists of. These rules are musically quite complex but mathematically less interesting, therefore we omit them; see Section \ref{modulations} of the Appendix for an overview.
As for the compositional principles not detailed in \cite{benson}, we follow the traditional Hungarian music theory coursebook \cite{kesztler}. We use the German notation of classical harmony, according to the Hungarian convention, but in the meaning of the axioms this makes no difference from e.g. the British notation system. 

The contribution of this paper is the construction of a mathematical axiom system for classical harmony. In particular, we classify musical keys. We prove that there are \emph{exactly three different keys} with the common base on the equal-tempered piano, up to enharmony. Our main result is the fundamental theorem of tonality, which makes it possible to embed the Classicist chord-changing compositional principles as constraints to a consistent axiom system. Our work gives a logical basis for writing a new music theory coursebook for high schools. This seems to be necessary in Hungary, and it can also be helpful in other European countries.

The paper is organized as follows. Section \ref{alja} enumerates and logically orders basic musical notions. It also explains compliance of triads and seventh chords with classical harmony. In Section \ref{hangnem}, we define and classify musical keys. In Section \ref{top}, we present the model of homophonic four-part pieces. Musical functions and functional tonality are defined in Section \ref{tonikus}, while the fundamental theorem of tonality and the structure of our axiom system is presented in Section \ref{fő}.

\section{Basic notions of music theory used in classical harmony} \label{alja}
We provide an axiom system for composing homophonic four-part pieces of music, in a first-order language. We use the language of set theory, assuming the Zermelo--Fraenkel--Choice (ZFC) axiom system. We use simple physical properties of the overtone system, but formally these only have arithmetic meaning. As usual in music theory, a \emph{tone} $Y$ is a longitudinal wave moving in an elastic medium with frequency $f(Y)>0$. For a tone $X$ with frequency $f(X)>0$, $X$ is \emph{audible} if $20~\mathrm{Hz}<f(X)<20~000~\mathrm{Hz}$. When speaking about tones, we always mean that the tone is uniquely determined by its frequency and consists of all of its \emph{overtones}. The set of overtones of the tone $X$ is $\lbrace Y \vert~ Y$ is a tone, $\exists n \in \mathbb{N}^+:~f(Y)=n~f(X) \rbrace$. The overtone of $X$ with frequency $n~f(X)$ is called the $n$th overtone of $X$. Thus, when we consider \emph{a tone $X$ with frequency $f(X)$}, $X$ can be mathematically described, e.g., as $X=(0,f(X)) \in (\mathbb{R}^+)^2$, thus the definitions of this paper can be derived consistently from ZFC\label{jujjujjuj}. In this sense, our \emph{compositional principles} which describe compliance of certain musical entities (e.g. triads, four-part pieces or modulations) with classical harmony are just special definitions in the mathematical framework of ZFC. But if one considers classical harmony in itself, they are indeed axioms, in the sense that these properties are postulated about all entities that comply with classical harmony. This is true for all Axioms in this paper. 

In the whole article, $B_r(x)$ denotes the open ball with radius $r$ around the point $x$ in any metric space, further $\overline{A}$ the closure of $A$ and $\partial A$ the boundary of $A$ in any topological space. $\mathrm{Int}$ denotes interior, $\mathrm{Dom}$ domain and $\mathrm{Ran}$ range.

Musical \emph{intervals} are equal distances in the (base 2) logarithmic frequency scale. The most important intervals can be derived from the overtone system, cf. \cite[Section 4.1]{benson}. The interval of a tone and its 2nd overtone is called \emph{perfect octave}, the one of the 2nd and 3rd overtone of a tone \emph{perfect fifth}, the one the 3rd and 4th overtone of a tone \emph{perfect fourth}. 
We say that a tone $X$ is higher than a tone $Y$ (and $Y$ is lower than $X$) if $f(X)>f(Y)$. Intervals can be summed, hence one can speak about \emph{octave-equivalent} tones $X$ and $Y$, the interval of which is $n$ octaves with $n \in \mathbb{Z}$. If $X$ is $n$ octaves higher than $Y$, this means $f(X)=2^n f(Y)$. Hence, octave equivalence is an equivalence relation on the set of tones. \emph{The octave-eqivalence class} of the tone $X$ will be denoted by $[X]$.
 
The following definitions will be used in Section \ref{tonikus}, where we define musical functions and tonality. 
Let $X$ be a tone and $Y$ its 3rd overtone. The \emph{leading tone} of $[Y]$ is the octave equivalence class of $X$'s 11th overtone; the \emph{seventh tone belonging to} $[X]$, \label{seventhtone}also called the \emph{upper leading tone of $X$'s fifth overtone's equivalence class} is the octave equivalence class of the 7th overtone of $Y$. We also define these relations for the tones themselves: e.g. for $U \in [U]$ and $V \in [V]$, if $[U]$ is the leading tone of $[V]$, then we say that $U$ is the leading tone of $V$.

We define the \emph{perfect $X_1$ major scale} for a tone $X_1$. Generally, a \emph{seven-degree scale with base $X_1$} is a set of tones $\lbrace X_1, X_2,\ldots,X_7 \rbrace$ where $f(X_i) > f(X_j) \Leftrightarrow i>j$ and $f(X_7)<2f(X_1)$ (this is, every member of the scale is strictly less then one octave higher than the base). $X_i$ is called the $i$th degree scale tone of the scale. According to the Hungarian notation, we denote the degrees and the operations among them with the elements of the prime field $\mathbb{Z}_{7}$ , but we use the capital Roman numeral for the integer $(n \pmod 7) +1$ instead of $n \in \mathbb{Z}_7$. The \emph{perfect $X$ major scale} is a seven degree scale with base $X$, where the frequency ratios of the neighbouring degree tones are respectively:
$\frac{9}{8}, \frac{10}{9}, \frac{16}{15}, \frac{9}{8}, \frac{10}{9},\frac{9}{8}, \frac{16}{15}$. 
where the last ratio is the ratio of the VIIth degree scale tone and the second overtone of $X$.
For a perfect major scale, the following are approximately true, in the sense that the human ear cannot observe that they are false:
\begin{enumerate}[(i)]
\item the degree VII tone is the leading tone of the degree I, III is the one of IV,
\item IV is the upper leading tone of III, I is the one of VII,
\item the interval between I and IV is a perfect fourth, the one between I and V is a perfect fifth,
\item IV is the seventh tone belonging to I, and I the one belonging to V.
\end{enumerate}
The sum of \emph{twelve perfect fifths} starting from a tone $X$ results a tone with frequency $\frac{531441}{4096} f(X)$, while the sum of \emph{seven perfect octaves} gives a tone with frequency $128~f(X)$. The difference of these two tones is noticeable by an average person. However, if this deviation is equally spread along the whole interval, it locally cannot be perceived. Therefore one aims to fease the concept of \emph{the circle of fifths}, i.e., 12 quasi-fifths equal to 7 octaves on a musical instrument each tone of which is a base of a seven-degree scale perceptually equivalent to a perfect major scale. The leading tone and upper leading tone/seventh tone connections between the quasi-perfect major scales could be used to make it possible to move from each major scale to the two with a base one quasi-perfect fifth higher respectively lower.
This is the idea of the \emph{equal-tempered piano}.
\begin{definition}
A countable set $K$ of tones is an \emph{equal-tempered piano} if
\begin{enumerate}[(i)]
\item $A \in K$, where $A$ is the normal $A4$ tone with frequency $440~\mathrm{Hz}$,
\item $K$ has two tones $X$ and $Y$ the interval of which is at least 7 octaves,
\item If a tone $X$ is an element of $K$, then $f(X)=(\sqrt[12]{2})^n~f(A)$ for some $n \in \mathbb Z$. Further, if $f(X)=(\sqrt[12]{2})^n~f(A)$ for some $n \in \mathbb Z$ and $\exists Y, Z \in K$ such that $f(Y) < f(X)<f(Z)$, then $X \in K$.
\end{enumerate}
\end{definition}
According to this definition, $A$ is the element of every -- finite or infinite -- equal-tempered piano. Thus, the octave equivalence classes of the piano's white keys ($A,B,\dotsc,G$) can be defined. Using a \emph{well-tempered}, i.e., approximately equal-tempered piano that already actualized the circle of fifths, J.S. Bach showed that every tone of the equal-tempered piano can serve as a base of a quasi-perfect major scale, by composing his \emph{Das wohltemperierte Klavier}, which contains one piece written in each major key of his well-tempered piano.

\emph{Enharmonic equivalence} in the context of the 12-tone equal tempered scale means that two tones $Y,Z$ originate from two different perfect major scales, but there is a tone $X$ on the equal tempered piano from which neither $Y$ nor $Z$ is significally different for the human ear. Enharmonic equivalence depends of the listener's own hearing and cultural background; here we follow the classicist European convention. Non-audible tones are called enharmonic if they have audible octave-equivalents that are enharmonic. The corresponding octave equivalence classes are also called enharmonic. If the tones $A$ and $B$ are enharmonic, we write $A \sim B$. It is easy to see that $\sim$ is an equivalence relation. Further,
\begin{enumerate}
\item The interval of two neighbouring tones of the equal-tempered piano is called \emph{semitone}, the sum of two semitones (the distance of second neighbours) a \emph{wholetone}. The sequence of tones 0, 2, 4, 5, 7, 9 and 11 semitones higher than an \emph{arbitrary} piano tone is enharmonic to a perfect major scale.
\item If $X$ and $Y$ are two piano tones where $Y$ is a wholetone higher than $X$, then the (only) piano tone $Z$ such that $f(X)<f(Z)<f(Y)$ is the leading tone of $Y$ and the upper leading tone of $X$, up to enharmonic equivalence.
\end{enumerate}
The C major scale on the equal-tempered piano consists of the seven white keys. Moving \emph{stepwise upwards} in the circle of fifths, one reaches the G, D, A, E, B major scales consecutively. At each step, one new tone appears in the scale, this is VIIth degree tone of the new scale. We denote this new tone by X$\sharp$, where X is the element of the C major scale which has been replaced by the new one semitone higher tone. This way the following tones appear consecutively: F$\sharp$, C$\sharp$, G$\sharp$, D$\sharp$, A$\sharp$, the leading tones to G, D, A, E, B respectively. After B, the next fifth step upwards leads to F$\sharp$. Now, starting from the C major scale again and move \emph{stepwise downwards} fifth by fifth, we reach the F major scale first, which has exactly one scale tone that is outside the C major scale: instead of B, a one semitone lower tone occurs: the seventh tone with respect to F. Let X$\flat$ denote the one semitone lower piano tone than the white key X, then moving downwards in the circle of fifths, we reach F, B$\flat$, E$\flat$, A$\flat$, D$\flat$ and G$\flat$ consecutively. $\sharp$ and $\flat$ marks can be multiplied. By construction, we have $\flat\sharp=\sharp\flat=\natural$ means the identity of the C major scale, and ``multiplication" of $\flat$'s and $\sharp$'s is commutative.

Note that G$\flat$ and F$\sharp$ refer to the same (black) piano keys, these two tones are enharmonic, also D$\flat$ is enharmonic to C$\sharp$ etc. However, they are the same only under equal temperament: if we build a \emph{perfect} A major scale, F$\sharp$ is the VI degree scale tone there, while G$\flat$ is reached if we move downwards in the D$\flat$ major scale by 8 \emph{perfect} fifth steps, and take the IV degree scale tone. It is a well-known experimental result that these actual G$\flat$ and F$\sharp$ differ significantly.

From this point, all major scales will be situated on an equal-tempered piano, with all degrees derivable from the C major scale with finitely many ---in practice, usually 0, 1 or 2 --- $\flat$s or $\sharp$s. The fifth-by-fifth sequence of sharpened scale tones of a major scale on the equal-tempered piano (F$\sharp$, C$\sharp$, $\ldots$) or the sequence of flattened scale tones of the major scale (B$\flat$, E$\flat$, $\ldots$) is called the major scale's \emph{key signature}. Having established these scales on the piano, the traditional notation of musical intervals among their degrees can be established, see e.g. \cite[Appendix E]{benson}. Also one can define \emph{consonance} and \emph{dissonance} of these intervals, cf. \cite[Chapter 4]{benson}.

By definition, an equal-tempered piano has to be \emph{at least as wide as a real piano}, in order to make it possible that the piano covers 7 octaves ($\approx$ 12 fifths). Let $d_2(X,Y)$ denote the interval of the notes $X$ and $Y$ of \emph{any} equal-tempered piano $K$, measured in semitones. The construction of equal temperament implies the following, also if the equal-tempered piano is infinite.
\begin{proposition}
$(K,d_2)$ is a metric space and $d_2$ generates the discrete topology. 
\end{proposition}
Thus, if $K$ is an equal-tempered piano, the Cartesian product $K^n$ is also equipped with the discrete topology. The elements of $K^n$ are called \emph{chords}.
Hence, we can speak about \emph{Borel-measurable functions} $M:~\mathbb{R}_0^+ \to K^n$, which we call \emph{$n$-part pieces}. The \emph{$k$th voice} of $M$ is $\mathrm{pr}_k \circ M$, where $\mathrm{pr}_k$ is the projection to the $k$th instance of the equal-tempered piano. We are interested in the \emph{four-part case}; we define the special, homophonic four-part pieces in Section \ref{top}. There the voices (in increasing order of their numbers) are called, as conventionally, \emph{bass, tenor, alto} and \emph{soprano}. In the simplest models of classical harmony, each chord appearing in a homophonic four-part piece has to be a \emph{triad} or a \emph{seventh chord} in an correct four-part form. In Bach's four-part chorales, this is not true any more, but most of the chorales can be obtained from a piece consisting of such chords via applying a finite set of local modifications, the so-called \emph{figurations}, see Definition~\ref{koralka} bel 

\emph{Triad names} are special elements of the factor space $K^3 / \equiv$ on an arbitrary equal-tempered piano $K$, where $\equiv$ is the octave equivalence relation. These contain scale tones or once altered tones from a certain seven-degree scale on $K$, and their main characteristic is that they consist of a $k$th, a $k+2$nd and a $k+4$th degree tone$\!\pmod 7$ of the given major scale based on one of the twelve enharmonic equivalence classes of the equal-tempered piano. With this notation, we say that the triad is of degree $k$. There are four kinds of triad names for which we say that they \emph{comply with classical harmony}, according to Table \ref{egykém}.

A \emph{four-part version of a triad} -- later in this article, simply: a \emph{triad} -- is an element of the piano power $K^4$, which consists of the tones of a triad name, exactly one of them in two voices. 
If the triad consists of the $k$th, $k+2$nd and $k+4$th degree scale tone of a seven-degree scale -- these are called the \emph{base}, the \emph{third} and the \emph{fifth} of the triad, respectively -- on the equal-tempered piano, its \emph{position} is determined by which tone it has in the bass. If in the bass there is the $k$ degree tone, where $k$ refers to the corresponding Roman numeral as before, then the triad is in \emph{root position} (German--Hungarian notation of the triad: $k$), if the $k+2$ degree tone is in the bass, then the triad is in \emph{first inversion} (notation: $k^6$), and if the $k+4$ degree tone, then in \emph{second inversion} (notation: $k_4^6$). (We note that in some models of classical harmony, triads without fifths are allowed, in such cases the condition that exactly one tone appears in two voices is not satisfied; either the base appears in three voices or both the base and the third are doubled. We will not take this into account in the rest of the present paper.)

\begin{table} \caption{Triads (above) and seventh chords (below)}\label{egykém} \begin{footnotesize} \begin{tabular}{|ll|ccc|}
\hline
Name & Notation & $k\leftrightarrow k+2$ interval & $k+2\leftrightarrow k+4$ i. & $k+4\leftrightarrow k$ i. \\ \hline
Major triad & M &\emph{major third} & minor third & perfect fifth \\ \hline
Minor triad & m & \emph{minor third} & major third & perfect fifth \\ \hline
Diminished triad & d & minor third & minor third & \emph{diminished fifth} \\ \hline
Augmented triad & A & major third & major third & \emph{augmented fifth} \\ \hline
\end{tabular}

\vspace{3pt}
\begin{tabular}{|l|l|l|l|l|l|l|}
\hline
Name: & Third & Fifth & Seventh & Partial triads & Examples & Example \\
($\ldots$) \emph{seventh} & & & & &  in major & in minor \\ \hline
augmented major &major & augmented & major & major, augm. &  none & III \\ \hline
major minor &major & perfect & major & major, minor &  I, IV & VI \\ \hline
major/dominant & major & perfect & minor & major, dimin. &  V & V \\ \hline
harmonic minor & minor & perfect & major & minor, augm. & none & I \\ \hline
minor major & minor & perfect & minor & minor, major & II, III, VI & IV \\ \hline
semi-diminished & minor & diminished & minor & dimin., minor &  VII & II \\ \hline
diminished & minor & diminished & diminished & dimin., dimin. & none & VII \\ \hline  
\end{tabular}
\end{footnotesize}
 \end{table}

Consider the union of a degree $k$ and a degree $k+2$ triad name on an arbitrary seven-degree scale. This is indeed an element of $K^4 / \equiv$, and it is called a \emph{seventh chord name}. If $H \in K^4$ consists of the tones of a seventh chord name in any permutation of the voices, then $H$ is called a \emph{seventh chord}. This name comes from the fact that there is a seventh interval between the $k$ and the $k+6$ degree scale tones. The degree $k$ and degree $k+2$ triads are the \emph{partial triads} of the seventh chord. The position of a degree $k$ seventh chord inverson can be: \emph{(root position) seventh chord} (German--Hungarian notation: $k^7$), \emph{first inversion} ($k_5^6$), \emph{second inversion} ($k_3^4$) and \emph{third inversion} ($k^2$), if in the bass there is the $k$th, $k+2$nd, $k+4$th and $k+6$th degree tone of the seventh chord name, respectively. For the origin of these notations, we refer to \cite[Book I., Section II.11]{kesztler}. 

The $k$th, $k+2$nd, $k+4$th and $k+6$th degree tones of a seventh chord are called \emph{base, third, fifth} and \emph{seventh} respectively. If both partial triads of a seventh chord name $H$ comply with classical harmony, and there is no triad which is voicewise enharmonic to $H$, then we say that $H$ \emph{complies with classical harmony}. This second assumption is taken for excluding the \emph{augmented triad} (see Table \ref{egykém}) from the set of seventh chord names, which can be represented as a seventh chord but is indeed just a triad. 
As a remark, we note that a root position \emph{dominant seventh} (see Table 1) that complies with classical harmony may be \emph{fifth deficient}, which means that it need not contain its fifth in any voice but instead the base in two voices (one of these voices is necessarily the bass). 

Table \ref{egykém} also shows the seventh chord types, with examples consisting of scale tones of the major and the minor key (see Section \ref{hangnem}). For four-part versions of triads and seventh chords, we say that they \emph{comply with classical harmony} if their name complies with classical harmony, in each of their voices the pitches (frequencies) accord to the conventional pitch interval associated with the instrument or singer that presents the voice, and the duplication of tones is correct. Questions about appropriateness of pitches in the four voices are quite delicate and they require music-historical research. From a model point of view, they can be ignored by just considering pieces on an equal-tempered piano that is infinite towards both directions. Classical tone duplication rules are e.g. that if a triad is in root position, then the only tone of the triad that may appear in two voices is the base etc. It is also required that all tones of a triad or seventh chord that are not scale tones of the scale corresponding to the chord may only appear in one voice. For detailed duplication rules, we refer to \cite[p.~30--183.]{kesztler}.

\section{Trichotomy of musical keys} \label{hangnem}
After introducing the basic musical notions, we define keys based on the idea of key stability and functional tonality in classical harmony. We preliminarily ensure that our definition accepts the major and the minor keys, which have been used in Europe for five centuries, to be keys. Furthermore, our key definition gives us the possibility to find all possible key types. Our Proposition \ref{hangnemtyű} shows that apart from major and minor there is exactly one more type.
\begin{definition} \label{hangnemke}
Let $H$ be a seven-degree scale on the equal-tempered piano (consisting of scale tones and altered tones from the C major scale), with seven pairwise non-enharmonic scale tones. We say that $H$ is a \emph{key} if:
\begin{enumerate}[(i)]
\item the Vth degree seventh chord of $H$ is dominant,
\item all triad and seventh chord names that consist of the scale tones of $H$ comply with classical harmony,
\item if the $k$th degree seventh is dominant, then the degree $k+3 \mod 7$ triad is major or minor$\pmod 7$, with the $k+3$th degree scale tone one perfect fourth higher than the $k$th degree one. 
\end{enumerate}
\end{definition}
The condition $(iii)$ means that the all dominant sevenths can \emph{resolve to their tonic}, see Section \ref{tonikus}. It follows that Definition \ref{hangnemke} implies the next two properties:
\begin{proposition}
In a key the first degree triad is major or minor, and the VIIth degree scale tone is the leading tone of the Ist degree scale tone.
\end{proposition}
The resolution of the Vth degree seventh (or triad) to the Ist degree triad is the key ingredient of the stability of the key, as detailed in Section \ref{tonikus}.
\begin{proof}
Parts (i) and (iii) of Definition \ref{hangnemke} imply that the Ist degree triad is major or minor, moreover that the Vth degree seventh chord must resolve to the Ist degree triad. Hence, the Vth degree scale tone is 7 semitones higher than the Ist degree one. Therefore, the VIIth degree scale tone, which is the third of the Vth degree triad, is one semitone lower than the Ist degree scale tone. This implies the second part of the claim.
\end{proof}
The next lemma is a key observation of this section.
\begin{lemma}[The Minor Lemma] \label{molllemma}
In any key $H$ the following are equivalent:
\begin{enumerate}[(i)]
\item the VIth degree scale tone is 8 semitones higher than the Ist degree one,
\item all types of seventh chords from Table \ref{egykém} can be built from scale tones of $H$,
\item the VIIth degree seventh chord (built from scale tones) is diminished.
\end{enumerate}
\end{lemma}
\begin{proof}
The definition of key implies that the sequence of intervals of the first degree scale tone and the other scale tones is: $(0,2,?,5,7,?,11)$ semitones. The ?'s refer to unknown intervals. It is easy to see that the conditions of the lemma are equivalent to the condition that the sequence of intervals is $(0,2,?,5,7,\mathbf{8},11)$. The remaining ? stands for either 3 or 4, in order to satisfy the definition of key.
\end{proof}
\begin{proposition}[The trichotomy of keys] \label{hangnemtyű}
Let $X$ be an enharmonic equivalence class on an equal-tempered piano $K$. Then there are exactly three keys with first degree $X$, up to enharmonic equivalence. These are the major, the minor and the harmonic major (named by Rimsky-Korsakov in \cite{rk}) keys, with interval sequences $(0,2,4,5,7,9,11)$, $(0,2,3,5,7,8,11)$ and $(0,2,4,5,7,8,11)$, respectively. The latter two ones are the ones that satisfy the Minor Lemma.
\end{proposition}
\begin{proof}
As in the proof of the Minor Lemma, the key's definition implies that the interval sequence of an arbitrary key's scale is $(0,2,?,5,7,?,11)$. If the degree VI scale tone has sign $9$, then the requirement that every triad and seventh chord built up from scale tones has to comply with classical harmony implies that the sign of the IIIrd degree tone is either $3$ or $4$. If this sign is $4$, then the scale is the $X$-major scale. If the sign is $3$, then the IVth degree seventh chord built up from scale tones is a dominant seventh, but the interval between the IVth and the VIIth degree scale tone is not a perfect fourth but a tritone (enharmonic with 6 semitones/3 wholetones). Therefore in this case we do not obtain a key.
\end{proof}

From the 1500s, European music is determined by the major--minor duality. The harmonic major key differs by only one scale tone (degree VI) from the major scale and also by only one scale tone (degree III) from the minor one, and therefore the listener automatically tries to perceive it as minor or major. This causes an instability of the harmonic major key in the context of European music history, which is however not an intrinsic instability of this key, since it satisfies the same stability conditions as the two other types of key, according to Definition \ref{hangnemke}.  

The \emph{key signature \label{kettke} of a harmonic minor or harmonic major key} is the key signature of the major scale which has the degree I scale tone of the minor key as its degree VI scale tone. These major and minor scales are called \emph{relative}.
\section{Topology of the homophonic four-part setting} \label{top}
In this section, we present a continuous time model of classical harmony, given by \emph{homophonic four-part pieces}. Our notions allow for some non-feasible musical phenomena, such as an infinite piece and chords accumulating in one point in time (referred to as \emph{packing point}). This way, one can also handle \emph{periodic pieces} without ending in finite time, which is often aimed in both classical (e.g. William Billings: \emph{The Continental Harmony}, 1794) and popular music. The continuous approach allows us to define the genre of \emph{Bach's chorale harmonizations} mathematically precisely, which is not possible if one uses only \emph{chord sequences} to describe homophonic four-part pieces. 
 
\begin{definition} \label{osszhangpelda}
Let $K$ be an equal-tempered piano. $M: \mathbb{R}_0^+ \to K^4$ be a four-part piece (see Section \ref{alja}). $M$ is called a \emph{homophonic four-part piece} if:
\begin{enumerate}
\item[(i)] Each element of $\mathrm{Ran}~M$ is a four-part version of a triad or a seventh chord (in some inversion) that complies with classical harmony, 
\item[(ii)] each voice of each element of $\mathrm{Ran}~M$ only contains tones that can be derived from the C major scale on $K$ using the system of $\flat$'s and $\sharp$'s,
\item[(iii)] for all $H \in \mathrm{Ran}~M$, we have that $M^{-1}(H)=\lbrace x \in \mathrm{Dom}~M \vert M(x)=H \rbrace$ is a disjoint union of intervals closed on the left and open on the right.
\end{enumerate} 
\end{definition}
\begin{definition}
If $M$ is a homophonic four-part piece, $B(M)$, the smallest (left-closed, right-open) interval that contains $\mathrm{Dom}~M$ is called the \emph{cover} of $M$.
\end{definition}
As we mentioned in the introduction, homophony means that if in a point in time one voice starts to play a new tone, then all other voices do so. It follows that if $M$ is a homophonic four-part piece, then if there is a pause at time $t \in B(M)$ in at least one voice of $M$ (i.e., $t \notin \mathrm{Dom}~pr_i \circ M$), then this is actually a \emph{general pause}, i.e. pause in all voices. Also, one can prove that the connected components of pauses are also intervals closed on the left and open on the right.

Using the point $(iii)$ of Definition \ref{osszhangpelda}, it is easy to see that any homophonic four-part piece $M$ is \emph{right-continuous}. 
According to the fact that $M$ takes values in a discrete space, this implies
$
\forall t_0 \in \mathrm{Dom}~M~\exists \delta>0: ~ \forall t \in \left[ t_0, t_0+\delta \right[ ~ M(t)=M(t_0).
$
Now we define some special points of homophonic four-part pieces.
\begin{definition}
Let $M$ be a homophonic four-part piece.
\begin{enumerate}
\item[(i)] $t=\inf~\mathrm{Dom}~M$ is the \emph{starting point} of $M$,
\item[(ii)] $t=\sup~\mathrm{Dom}~M$ is the \emph{endpoint} of $M$, 
\item[(iii)] $t \in \mathrm{Dom}~M$ is a \emph{chord-changing point} of $M$ if $\exists \varepsilon>0,~\exists H_1 \neq H_2 \in K^4$ such that $\forall x \in \left[ t-\varepsilon, t \right[~ M(x)=H_1$ and $\forall x \in \left[ t, t+\varepsilon \right[,~M(x)=H_2$. From now on, let $A(M)$ denote the set of the chord-changing points of $M$.
\end{enumerate}
\end{definition}
In the following, $\vee$ means logical ``or" and $\wedge$ means logical ``and".
\begin{definition}
A homophonic four-part piece $M$ has \emph{infimum of chord lengths} defined as  \[ \inf\limits_{t \in \mathcal{D}(M)} \sup \left\lbrace r_1+r_2 \vert r_1, r_2 \geq 0 ~ \wedge \forall x \in \left[ t-r_1, t+r_2 \right[: M(x)=M(t) \right\rbrace . \] 
\end{definition}
The proof of the next proposition is left for the reader.
\begin{proposition}
Let $M$ be a homophonic four-part piece. Then $ A(M)$ is countable. If the infimum of the chord lengths of $M$ is positive, then $A(M)$ has no accumulation point (i.e., $\overline{A(M)}$ consists only of isolated points), which implies that $A(M)$ is finite if $\mathrm{Dom}~M$ is bounded.
\end{proposition}
The following definition accounts for a non-feasible musical effect coming from our topological model, which is playing infinite music in finite time.
\begin{definition}
Let $M$ be a homophonic four-part piece. $t \in \overline{\mathrm{Dom}~M}$ is a \emph{packing point} of $M$ if $\forall \varepsilon>0$ $\left[ t-\varepsilon, t \right[$ contains infinitely many chord-changing points or \emph{isolated} boundary points of $\mathrm{Dom}~M$.
\end{definition}
Packing points can have interesting applications in the spirit of \cite[part~I.;~Decision]{ligeti}. But in the usual model in classical harmony, packing points do not occur, and neither do pieces without packing points but with infinitely many chords.
\begin{definition}
A homophonic four-part piece $M$ is \emph{feasible} if
\begin{enumerate}
\item[(i)] the infimum of chord lengths of $M$ is positive, and if $\mathrm{Dom}~M \neq B(M)$, then the infimum of general pause interval lengths is also positive,
\item[(ii)] and $\overline{\mathrm{Dom}~M}$ is compact.
\end{enumerate}
\end{definition}
\section{Definition of Bach's chorales}
In order to mathematically define chorales, first we introduce \emph{playing functions}, which describe the performance of these pieces with non-constant velocity.
\begin{definition}
Let $M$ be an $n$-part piece for a certain $n \in \mathbb{N}^+$ and $\theta:~ \left[ 0, \infty \right[ \to \left[ 0, \infty \right[$ a continuous, strictly increasing function, for which $\left[ 0, \infty \right[$ can be divided into countably many disjoint intervals $(I_i)_{i \in \mathbb N}$ joining each other and altogether covering $\left[ 0, \infty \right[$, such that restricted to the interior of each interval $I_i$, $\theta$ is twice continuously differentiable, $\theta'$ nowhere vanishes and $\inf\limits_{n \in \mathbb N} \lambda(I_n)>0$. \\ Then $\theta$ is called a \emph{playing function}. The name of $M \circ \theta \vert_{B(M)}$ is the \emph{playing} of $M$ that belongs to $\theta$. For $t \in B(M)$, $\theta'(t)$ is called the \emph{playing velocity} and $\theta''(t)$ the \emph{playing acceleration} in the point $t$, if they exist. $\theta \equiv 1$ gives the \emph{naturally parameterized} $n$-part piece. The set of playing functions is denoted as $PL(\mathbb{R^+})$.
\end{definition}
It is easy to verify that on bounded intervals playing functions are absolutely continuous. The following definition is based on this.
\begin{definition}
If $\theta \in PL(\mathbb{R})$, $M$ is a homophonic four-part piece and $A$ is a Lebesgue-measurable subset of $B(M)$, then the \emph{length} of the part $A$ of piece $M$ by the playing function $\theta$ is $\mu_{\theta}(A)=\int\limits_{A} 1 \theta(\mathrm{d} x) = \int\limits_{A} \theta'(x) \mathrm{d}x$. 
\end{definition}
The proof of the next proposition is left for the reader.
\begin{proposition}
$PL(\mathbb{R})$ is a group under the composition of playing functions.
\end{proposition}
Now we turn our attention to the mathematical definition of the genre of Bach's chorales. We emphasize that chorale is an actual musical genre from Baroque, and hence its characteristics are originally non-prescriptive. Therefore however precisely we define a chorale, our definition may only be correct for the majority of the pieces, with certain exceptions. 
\begin{definition}
Let $M$ be a homophonic four-part piece with $x \in \mathrm{Dom}~M$ and $M(x)=H$. Then the \emph{area} of $M(x)$ is the connected component of $M^{-1}(H)$ containing $x$. \\
The \emph{halving} of $I=\left[ a, b \right[ \subseteq B(M)$ in the playing belonging to $\theta \in PL(\mathbb{R})$ is dividing $I$ into two disjoint intervals closed on the left and open on the right $I_1, I_2$ which together cover $I$ and $\mu_{\theta}(I_1)=\mu_{\theta}(I_2)$.
\end{definition}
Using this, our definition for four-part chorale is the following. 
\begin{definition} \label{koralka}
A four-part piece $\mathfrak{K}$ is a \emph{four-part chorale} if there is a feasible, naturally parameterized, pauseless homophonic four-part piece $M$ such that $\exists c>0$: $\forall x \in \mathrm{Dom}(M)$ the length of the area of $x$ by the identic playing function is $c$, and \\
$(1)$ $\mathfrak{K}$ can be derived from $M$ with using the following steps, the so-called \emph{figurations}. They are used for a finite number of $x \in \mathrm{Dom}~M$ and the figurations excluding each other are not done at the same time. \\ Types of the figurations are: \begin{description}
\item[Chord duplication] Halve the area of $M(x)$ by the identic playing function (natural parametrization), and in the first half of the area keep $M(x)$ for $\mathfrak{K}(x)$, in the other half $\mathfrak{K}(x)$ is one constant triad or seventh chord different from $M(x)$.
\item[Suspension] Halve the area of $M(x)$ by natural parametrization, in the second half of the area keep $M(x)$, in the first half, in one or two voices change the appropriate tone of $M(x)$ one step higher and keep the remaining voices.
\item[Advancement] Halve the area of $M(x)$ by natural parametrizaton, in the first half of the area keep $M(x)$, in the second half, in exactly one voice write a step higher or lower tone, which is equal to the tone in the same voice of the next chord after $M(x)$.
\item[Accented passing tone] Suppose that there is a third skip in some voice(s) of $M$ at arriving at or departing from $M(x)$. Then halve the area of $M(x)$, and on the half which is closer to the interval of the neighbouring chord that is involved in the third skip, instead of $M(x)$, write a tone the degree of which is between these two tones' degree. Note that only one accented passing tone per one chord area of $M$ is accepted.
\end{description} 
$(2)$ The given four-part piece $\mathfrak{K}$ is meant to be associated with a canonic playing function $\theta$ that differs from the identic playing in the following: $\exists m \in \mathbb{N}$ such that $\theta$ changes the length of every $m$th chord interval of $M$ to $k$ times greater than originally, where $k \in \left] 1, 2 \right[$ is a conventionally accepted factor. Then we say that there is a \emph{pause} on every $m$th metric unit.
\end{definition}

\section{Convergence area of a key. Functions and tonality} \label{tonikus}
Let $T$ be a key with degree I scale tone $X$ in an enharmonic equivalence class on the equal-tempered piano $K$. The \emph{convergence area} of $T$ is defined $CA(T)=\lbrace (G_i, L_i) \vert i=1,\ldots,N \rbrace$, where $N\in\mathbb{N}$ and $\forall i$, $G_i$ is a fixed triad or seventh chord name of $K$ in a certain inversion and $L_i$ is the list of the accepted four-part versions of $G_i$, according to the definitions of compliance with classical harmony from Section \ref{alja}. In simplified notation, we also call $G_i$ an element of the convergence area, and view $CA(T)$ as the set of chords belonging to $T$. 

We explain the meaning of convergence area of $T$ as follows. Roughly speaking, a chord $X$ that complies with classical harmony and is deducable from the C major scale on $K$ by $\sharp$s and $\flat$s is considered to be an element of $CA(T)$ if:
\begin{enumerate}
\item it is built from the scale tones of $T$ and regularly used in homophonic four-part pieces associated to this key. Some inversions of some chords are excluded for their too strong dissonances, e.g. diminished triads may only stand in first inversion, moreover in the case of degree VII triads, the duplicated tone of the triad must be the third. Apart from $\mathrm{I}^6_4$ and $\mathrm{IV}^6_4$, triads in third inversions are not used. All such seventh chords, which are called \emph{diatonic seventh chords}, are used in all inversions, though some of them very rarely. 
\item if $X \in CA(T)$ has a tone outside scale of $T$, then $X$ is called an \emph{altered chord} of $T$. Such $X$ is convergent if and only if it can lead to chords in $CA(T)$ built from scale tones, without violating any chord-changing compositional principles, in such a way that was usual in the practice of Viennese classical composers. A full list of convergent chords in minor and major key can be found in Section \ref{durmoll} of the Appendix.
\end{enumerate}

After introducing convergence areas, we define weak tonality.
\begin{definition}
Let $M$ be a homophonic four-part piece and $t$ an accumulation point of $\mathrm{Dom}~M$. We say that $M$ is weakly tonal in the point $t$ with key $T$ if there is a connected open neighbourhood $U$ of $t$ such that $\forall x \in (U \setminus \lbrace t \rbrace) \cap \mathrm{Dom}~M$, $M(x)$ is the element of $CA(T)$.
\end{definition}
Weak tonality may be sufficient in the case when there are no modulations among different keys, but classical harmony has stronger measures on key stability, especially for establishing new keys after modulations. This involves the notion of musical \emph{functions}: the \emph{tonic, dominant and subdominant}.
In the following, the $k$th degree triad or seventh chord of the key $T$ of a tone will mean the one built from scale tones of $T$. The \emph{leading tone/seventh tone of a key $T$} will refer to the leading tone/seventh tone of the key's Ist degree scale tone. \emph{The leading tone of a diminished triad or diminished seventh} is, by definition, its base.
\begin{definition}
Let $T$ be a key, $H \in CA(T)$ be a major or diminished chord, i.e. major triad, diminished triad, major (dominant) seventh or diminished seventh, and $G \in CA(T)$ a major or minor third. We say that $H$ \emph{resolves to} $G$ if \begin{enumerate} \item[(i)] $H$ contains the leading tone of (the base of) $G$, and \item[(ii)] if there is a tone $x$ belonging to $H$ that is not a scale tone in the major key built on the base of $G$, then $x$ is the upper leading tone of the fifth of $G$. \end{enumerate}
\end{definition}
\begin{definition}[Dominant function (D) and secondary dominant property.]
$X \in CA(T)$ has the \emph{dominant function} in the key $T$ if it resolves to the first degree triad of $T$. $Y \in CA(T)$ is a \emph{secondary dominant chord} if it resolves to any other major or minor chord built from the scale tones of $T$.  
\end{definition}
\begin{definition}[Tonic function (T)]
$X \in CA(T)$ has the \emph{tonic function} in the key $T$ if
\begin{enumerate}[(i)]
\item $X$ contains a Ist and IIIrd degree tone of $T$, the first one from the scale $T$,
\item if $X$ contains the leading tone of $T$, then it is the seventh tone of $X$,
\item if $X$ is secondary dominant, then $X$ is a Ist degree major triad,
\item $X$ has no augmented and no diminished partial triad.
\end{enumerate} 
\end{definition}
\begin{definition}[Subdominant function (S)]
$X \in CA(T)$ has the \emph{subdominant function} in the key $T$, if
\begin{enumerate}[(i)]
\item $X$ contains the IVth and VIth degree scale tone of $T$, possibly both altered,
\item if $X$ is secondary dominant (itself, not just up to enharmonic equivalence), then it resolves to the Vth degree triad. Moreover, the VIth degree tone of $X$ has neither more $\sharp$s nor more sharpening $\natural$s than the key signature of $T$. 
\item The set of tones of $X$ and the one of the Ist degree seventh chord of $T$ have no other common tone than the Ist degree scale tone. Moreover, $X$ contains no altered Ist degree tone.
\end{enumerate}
\end{definition}
In Table \ref{lúfej}, we present the most typical tonic, dominant and subdominant chords built of scale tones. Here, if a root position triad participates in the table, then its first inversion has the same function. A seventh chord participating in the list has the same function as any of its inversions. Note that we require that the \emph{seventh degree diminished seventh chord}, is also the element of $CA(T)$ if $T$ is a major key, where it is an altered chord. The Minor Lemma (Lemma \ref{molllemma}) guarantees that this chord is built from scale tones in a minor or harmonic major key but not in a major key. The non-altered convergent chords which are not listed in the table have no certain function. E.g., this applies for the IIIrd degree triad, as it is considered to be pending between tonic and dominant function, and many diatonic seventh chords also do not have a certain function.
\begin{table}
\centering
\begin{small}
\caption{Chords belonging to the three functions in the three different kinds of keys.}\label{lúfej}
\begin{tabular}{|c|c|c|c|} 
\hline
Type of $T$ & Tonic chords & Dominant chords & Subdominant chords \\ \hline
major & $\mathrm{I}$, $\mathrm{VI}$, $\mathrm{VI}^7,$ $\mathrm{I}^7$ & $\mathrm{V}$, $\mathrm{VII}^6$, $\mathrm{V}^7$ & $\mathrm{II}$, $\mathrm{IV}$, $\mathrm{II}^7$ \\ \hline 
minor & $\mathrm{I}$, $\mathrm{VI}$, $\mathrm{VI}^7$ & $\mathrm{V}^{\sharp}$, $\mathrm{VII}^{6\sharp}$, $\mathrm{V}_{\sharp}^7$, \small{$\underset{\sharp}{\mathrm{VII}}^7$} & $\mathrm{II}$, $\mathrm{IV}$, $\mathrm{II}^7$ \\ \hline
harmonic major & $\mathrm{I}^{\sharp}$, $\mathrm{VI}^{5\sharp}$, $\mathrm{VI}_{5\sharp}^7$ & $\mathrm{V}^{\sharp}$, $\mathrm{VII}^{6\sharp}$, $\mathrm{V}_{\sharp}^7$, \small{$\underset{\sharp}{\mathrm{VII}}^7$} & $\mathrm{II}$, $\mathrm{IV}$, $\mathrm{II}^7$ \\ \hline
\end{tabular}
\end{small}
\end{table}
\label{funky}

The degree I triad is called the \emph{tonic main triad} of the key $T$, the degree IV one is the \emph{subdominant main triad} and the degree V one is the \emph{dominant main triad}. \emph{Authentic step} means two different things in classical harmony. On the one hand, modulation (key change) to the one fifth higher key (the \emph{dominant} key), without changing the type of key. Among triads this means a V$\to$I or I$\to$IV type chord progression. On the other hand, function change $D \to T$ in a certain key in general. Similarly, \emph{plagal step} means two things. On the one hand modulation to the one fifth lower ---the \emph{subdominant}--- key, and among triads making a I$\to$V or IV$\to$I step, on the other hand function change $T \to D$ in certain key. 

A \emph{cadence} is a chord progression consisting of at least two chords that is considered to be appropriate for finishing a piece. In view of this, in certain cases, we will call the dominant$\to$tonic and tonic$\to$subdominant steps \emph{authentic cadences} and the tonic$\to$dominant and subdominant$\to$tonic steps \emph{plagal cadences}. In a given key, a \emph{complete authentic cadence} is a chord progression with $T \to S \to D \to T$ function sequence, while a \emph{complete plagal cadence} is a chord progression with $T \to D \to S \to T$. It is a well-known fact that complete authentic cadences are the most applicable for finishing a piece, for which there are many arguments, but it is hard to get a full explanation, cf. \cite[Section 5.11]{benson}. Most of the classical, romantic and also recent popular music is based on $D \to T$ resolutions, supported by complete authentic cadences using the function $S$. 

In the following, we present local and global notions of strong, \emph{functional tonality}. We already have all notions that we need in order to define tonality in a given point. The idea of this definition is to assign a key $T$ to the point as a limit, requiring that all three functions of $T$ occur in the vicinity of the point.
\begin{definition}[Local tonality with a given key.] \label{tonal}
Let $M$ be a (not by all means homophonic) four-part piece and $t$ an accumulation point of $\mathrm{Dom}~M$. $M$ is \emph{tonal} in the point $t$ \emph{with key} $T$ if there is a connected open neighbourhood $U$ of $t$ such that $V=(U \setminus \lbrace t \rbrace) \cap \mathrm{Dom}~M$, $M(x)$ satisfies the following conditions:
\begin{enumerate}
\item[(i)] $M$ is weakly tonal with key $T$ in every point of $V$,
\item[(ii)] $M[V]=\lbrace M(x) \vert~x \in V \rbrace$ contains at least one chord from \emph{all functions} of $T$,
\item[(iii)] if $t \notin \mathrm{Int~Dom}~M$, then only triad-valued points of $\mathrm{Dom}~M$ accumulate to $t$.
\end{enumerate}
\end{definition}

In order to define tonality of an entire piece, we need to provide our first axiom of classical harmony, in particular about \emph{modulations}. 

\begin{definition}[Modulation] \label{modka}
If there are keys $T_1$ and $T_2$ for the homophonic four-part piece $M$ such that $\mathrm{Dom}~M$ has a subset $Z=\left[ a, b \right[$, for which $M |_Z$ is feasible, and $\exists r_1>0, r_2>0$ such that on the whole set $(B_{r_1}(a) \cap \mathrm{Dom}~M) \setminus Z$, $M$ is weakly tonal with key $T_1$ and on the whole set $(B_{r_2}(b) \cap \mathrm{Dom}~M) \setminus Z$, $M$ is tonal with key $T_2$, then $\forall W \subseteq Z$ we say that $W$ belongs to a $T_1 \to T_2$ \emph{modulation}. We also say that $M$ \emph{modulates} on $Z$ from $T_1$ to $T_2$.
\end{definition}
Thus, we demand that modulations themselves be \emph{feasible} and \emph{pauseless}: they need to last until a finite time, without general pauses and packing points.
\begin{definition}[First modulational axiom] \label{modi}
Let $M$ be a homophonic four-part piece. If $M$ complies with classical harmony and contains a $T_1 \to T_2$ modulation, then $\exists~\left[ a, b \right[ \subseteq \mathrm{Dom}~M$ such that $M(a)$ is the degree I triad of $T_1$(built from scale tones), $M(b-)=\lim_{x \downarrow b}$ is the degree I triad of $T_2$ (also consisting of scale tones), and $M$ is weakly tonal in $a$ with key $T_1$, $M$ is tonal (in the sense of Definition~\ref{tonal}) in $b$ with key $T_2$, and $\left[ a, b \right[$ is the largest interval which belongs to this $T_1 \to T_2$ modulation.  
\end{definition}
Note that while the modulation can only be finished in a correct way if tonality in the new key is provided, at the beginning of the modulation only weak tonality in the old key is required. Indeed, there are well-known examples of homophonic four-part pieces consisting of one schematic modulation that do not satisfy Definition~\ref{tonal} in their starting point with the starting key.

\begin{definition}[Local tonality via modulation.] \label{tonalmod}
Let $M$ be a homophonic four-part piece, $t$ an accumulation point of $\mathrm{Dom}~M$ and $T_1 \neq T_2$ two keys. $M$ is \emph{tonal} in $t$ and \emph{modulates from} $T_1$ to $T_2$ if there is a connected open neighbourhood $U$ of $t$ such that $\exists \left[a, b \right[=V \supseteq U$, where $V$ belongs to a modulation (see Definition \ref{modka}), which complies with classical harmony apart from the chord-changing points.
\end{definition}
According to this, we define tonality as a global property of a piece as follows.
\begin{definition}[Tonal piece]
Let $M$ be a homophonic four-part piece, $A \subseteq \overline{\mathrm{Dom}~M}$. $M$ is tonal on $A$ if $\forall x \in \overline{A}$, $M$ is tonal in $x$ by Definition \ref{tonal} or \ref{tonalmod}.
\end{definition}

\section{Axioms and the fundamental theorem of tonality} \label{fő}
The most well-known classical compositional principles are the \emph{chord-changing or voice-leading rules}. The goal of the homophonic four-part model is to describe the kind of chord progression and voice-leading between chords that classical harmony accepts. Often the formal rules of classical harmony do not tell how to write pieces but what to avoid: it forbids some kinds of chord progressions (e.g. V$\to$IV steps in some cases) and some kinds of voice-leading (e.g. parallel octaves or augmented second steps). Virtually, this property of the axiom system gives the freedom to actually write \emph{pieces of art} and not just ‘correct examples' complying with classical harmony. What one \emph{should} write follows from the practice of Viennese classicist authors, cf. \cite[Section 5.11]{benson}. 
Now, we present our axiom system, which consists of:
\begin{itemize}
\item the rules of compliance of triads and seventh chords with classical harmony,
\item the definition of correctable piece (global level of the pieces),
\item the compositional principles for modulations (semi-global level, describing global properties of a modulational segment of a piece). For these, see Section \ref{modulations} of the Appendix or \cite[p.~36]{statusquo},
\item the compositional principles for chord-changes (local level).
\end{itemize} 

Our main result, the \emph{fundamental theorem of tonality} helps us embed the chord-changing rules in a mathematical axiom system for classical harmony. Its equivalent condition for tonality gives a general framework according to which pieces can comply with classical harmony \emph{apart from the chord changes}.
\begin{definition}[Correctable piece] \label{korrekt}
Let $M$ be a homophonic four-part piece, $a \in \mathrm{Dom}~M,~ b \in \overline{\mathrm{Dom}~M} \cup \lbrace \infty \rbrace, ~N =\left[ a, b \right[ \subseteq B(M)$. $M \vert_N$ is called a \emph{correctable piece} if all the following conditions are satisfied
\begin{enumerate}
\item[(i)] $M \vert_N$ has a positive infimum of chord lengths,
\item[(ii)] $N$ is the disjoint union of a finite even number $2n+2~(n \geq 0)$ of left-closed, right-open intervals $(I_0, I_1, \ldots, I_{2n+1})$ such that $\forall i \in \lbrace 1, 2, \ldots 2n+1 \},~ I_i \cap \mathrm{Dom}~M \neq \emptyset$, and either $I_0=\emptyset$ or also $I_0 \cap \mathrm{Dom}~M \neq \emptyset$. Further, $\forall 0 \leq k \leq n$ such that $I_{2k} \neq \emptyset$, in the whole interval $I_{2k}$, $M$ modulates complying with classical harmony apart from the chord-changing points, and $\forall 0 \leq k \leq n$ for $I_{2k+1} \cap \mathrm{Dom}~M$ there is a unique key $T_k$ such that $\forall G \in M[I_{2k+1}]= \lbrace H \in K^4 \vert \exists x \in I_{2k+1}: M(x)=H \rbrace$, $G \in CA(T_k)$, and $M[I_{2k+1}]$ contains at least one tonic, one dominant and one subdominant chord of $T_k$,
\item[(iii)] If $x \in N \cap \partial M$, then only triad-valued points accumulate to $x$.
\end{enumerate}
Moreover, if $t \in \mathrm{Int~Dom}~M \cap N$, then if there is no chord change in $t$ forbidden by the axioms regarding chord changes (see below), then we say that $M$ \emph{complies with classical harmony} in $t$ and $N$ is a \emph{classical neighbourhood} of $t$. \\
Finally, if $t_0 \in M$ is such that $t_0 \in \mathrm{Dom}~M$ but $\exists \varepsilon>0$ such that $t_0$ is the starting point of $M|_{N \cap [t_0-\varepsilon,\infty)}$, then according to the definition of strict four-part setting, $\lim\limits_{t \to t_0+0} M(t)=M(t_0)$. Then, if $\exists r>0$ such that $\forall x \in \left] t_0, t_0+r \right[: x \in N$ and $M$ complies with classical harmony in $x$, then we say that $M$ \emph{complies with classical harmony} in $t_0$. If $t_0 \neq a$, then we call $N$ a \emph{classical neighbourhood} of $t_0$. 
\end{definition}
\begin{definition}
If a homophonic four-part piece $M$ complies with classical harmony in $x$, $\forall x \in \mathrm{Dom}~M$, then we say that $M$ \emph{complies with classical harmony}.  
\end{definition}
The intuition behind these two definitions is the following. A piece that complies with classical harmony shall consist of segments that exhibit functional tonality with a given key, and sufficiently regular modulations leading from one such segment to the next one. The piece shall end with tonality in some key: the \emph{final key} of the piece, which often gives the name of the piece in Classicist music (e.g. String quartet in D major etc.)
\begin{theorem}[The fundamental theorem of tonality] \label{vécsey}
Let $M$ be a homophonic four-part piece that is pauseless ($\mathrm{Dom}~M=B(M)$) and feasible. Then $M$ is tonal (on $\overline{\mathrm{Dom}~M}$) if and only if it is correctable, i.e. if it complies with classical harmony (on $\mathrm{Dom}~M$) apart from its chord-changing points.
\end{theorem}
\begin{proof}

The fact that the condition of the theorem is sufficient for the tonality is almost clear from Definitions \ref{tonal} and \ref{korrekt}, therefore we omit this part of the proof.

We show that the condition is necessary for the tonality. Let $M$ be tonal, feasible and pauseless. For all $t \in \overline{\mathrm{Dom}~M}$, let $U_t$ be an open neighbourhood of $t$ that shows its tonality. If possible, let us choose $U_t$ such that it shows key and not modulation. Since $\mathrm{Dom}~M$ is a bounded subset of $\mathbb{R}$, it can be assumed that $\forall t \in \overline{\mathrm{Dom}~M}$, $U_t$ is a bounded open interval. Then, $\bigcup_{t \in \overline{\mathrm{Dom}~M}} U_t$ is an open cover of $\overline{\mathrm{Dom}~M}$. Since $\overline{\mathrm{Dom}~M}$ is compact, it has an open subcover, which we denote by $\bigcup_{i=1}^n U_i$. Without loss of generality, we can assume that $U_i= \left] a_i, b_i \right[$, where $a_i<a_j \Leftrightarrow i<j$ and $b_i < b_j \Leftrightarrow i<j$, moreover that $\inf~U_1$ is the starting point of $M$ and $\sup~U_n$ is the endpoint of $M$. The tonality of $M$ guarantees that there are two cases. The first one is that in $\inf~U_1$, $M$ is tonal with some key $T_1$. In this case, let us start a sequential process with $V=U_1$, $\mathfrak I=\emptyset$, $j=1$ and $T=T_1$ in order to divide $B(M)$ into an interval system that shows that $M$ is correctable. 

1. If $\forall k>j$ we have that in $U_k$, $M$ is tonal with same key as in $V$ (or $\sup V$ is the endpoint of $M$), then let us append $(V \cup \bigcup_{k>j} U_k)\cap \mathrm{Dom}~M$ to $\mathfrak I$, as the next interval for showing correctability. In this interval $M$ is tonal with key $T$. Also, only triad-valued points of $\mathrm{Dom}~M$ accumulate to the only two boundary points that $\mathrm{Dom}~M$ has, which are the starting point and the endpoint.

2. Else, $\exists k>j$ such that in $U_k$ there is a key $T'$, since Definition \ref{tonal} implies that in $\sup\mathrm{Dom}~M$ there has to be tonality with a key. Then let us define $s= \sup \lbrace x \in \left] \inf V, \sup U_k \right[ \vert~M\text{ is tonal in x with key }T\rbrace$ and $i=\inf \lbrace x \in \left] \inf V, \sup U_k \right[ \vert~M\text{ is tonal in x with key }T'\rbrace$. Note that $s$ and $i$ are finite, and the tonality of $M$ implies weak tonality with key $T$ in $s$ and weak tonality with key $T'$ in $i$, therefore $s \leq i$. By Definition \ref{modi}, $s<i$ follows. Then, on the entire interval $\left[ s, i \right[$, $M$ modulates from $T$ to $T'$ complying with classical harmony apart from the chord-changing points. Let us append $\left[ \mathrm \inf~ V, s \right[$ (as an interval with key $T$) and $\left[ s, \sup~ U_k \right[$ (as an interval of a $T \to T'$ modulation) to the set $\mathfrak I$ of the intervals showing the correctability of $M$. Let us put $T=T'$, $j=k$ and $V=\left[ i, \sup U_k \right[$ and return to the starting alternative of the sequential process.  

Each time the process restarts, the endpoint of the current $V$ is the endpoint of $U_k$ for a $k$ larger by at least 1 than the one in the previous turn. This ensures that the process is finite, further the number of turns is not more than $n$: when $\sup V=\sup \mathrm{Dom}~M$ holds, the process is finished. The intervals given by the process show that $M$ is correctable: the intervals with an odd index are intervals where $M$ has a key and the ones with an even index contain modulation from the previous interval's key to the following one's. Thus, each point of $\mathrm{Dom}~M \setminus A(M)$ has a classical neighbourhood containing $B(M)$. 

The second case is that there is no tonality with any key in $\inf~U_1$. In this case, by Definition \ref{korrekt}, the whole interval $U_1$ belongs to a modulation from some key $T_1$ to another one $T_2$. Then let $s:=\inf \{ x \in U_1 | ~M\text{ is tonal on $[x,U_1[$ with key $T_2$} \}$ and $I_0:=[\inf~U_1,s[$. Now, $M$ restricted to $\mathrm{Dom}~M \setminus I_0$ is such that it is tonal in its starting point with a key, and thus the intervals $I_1,\ldots,I_{2n+1}$ can be constructed similarly to the case when $M$ is tonal in its starting point. This finishes the proof. \end{proof}

We note that each condition of the theorem is necessary, i.e. for each one of them, one can find a tonal piece $M$ that violates it and therefore is not correctable.
\begin{itemize}
\item A tonal piece may have a packing point.
\item A tonal piece $M$ with $B(M)=\mathbb{R}_0^+$ may have no packing point but the infimum of the lengths of chord intervals can be still zero (in this case the sum of the chord lengths must be infinite).
\item A tonal piece $M$ with $B(M)=\mathrm{R}_0^+$ and positive infimum of chord lengths can contain infinitely many modulation intervals. In this case it can occur that $\forall t \in \left] \inf \mathrm{Dom}~M,\infty \right[$, we have that $M \vert_{\left[ 0, t \right[}$ is correctable but $M$ itself is not correctable. In this case, a \emph{final key} of $M$ cannot be defined. The final key is a main characteristic of finite feasible tonal pieces, in particular, in classicist music the names of the pieces of music often contain the final key of the piece (e.g. symphony in G major etc.).
\item A tonal, feasible piece is not by all means correctable if it is not pauseless. Indeed, take a piece that complies with classical harmony and contains no modulations, write some positive amount of it, and then continue with the same piece in a different key. The two connected components of the resulting piece with classical harmony by themselves comply , but the entire piece is not correctable because the modulation between the two keys is missing. For an example for such a piece, see Figure \ref{rosszpelda0}.
\end{itemize}
\begin{figure} \label{rosszpelda0}
    \centering
    \includegraphics[scale=0.4]{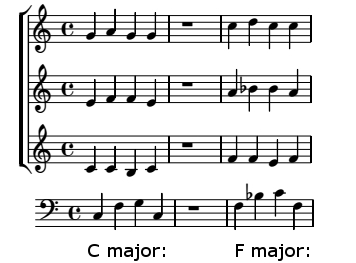}
    \caption{A non-pauseless tonal piece that is not correctable due to lack of modulation.}
\end{figure}

Knowing the fundamental theorem of tonality, we can embed the basic chord-changing compositional principles of classical harmony (see e.g. \cite[p.~30--183.]{kesztler} or \cite[Section 5.11]{benson}) in our mathematical axiom system. We present a scheme and an example how these compositional principles can be stated, knowing Theorem \ref{vécsey}.
\begin{definition}[Scheme of chord-changing rules] \label{séma}
Let $M$ be a tonal, feasible, pauseless homophonic four-part piece and $t$ a chord-changing point of $M$. If $M$ complies with classical harmony in $t$, then [\emph{conditions on the chord change in $t$}]. 
\end{definition}
\begin{definition}[Prohibition of parallel octaves.]
Let $M$ be a feasible tonal four-part piece and $t$ a chord changing point of $M$ with $M(t_-)=A$, $M(t)=B$. If the interval of two voices of $A$ is an integer number of octaves, then the interval of the corresponding voices of $B$ is not the same interval.

\end{definition}
This scheme guarantees that new chord-changing compositional principles can be added to the axiom system of classical harmony as long as the compositional principles do not contradict each other. The general historic experience is that the set of four-part rules of classical harmony is \emph{consistent}, equivalently that the axiom system actually has a model. In the full Hungarian version of the paper, we precisely described the classical chord-changing principles and the modulational rules. As for a model, we provided examples that satisfy various combinations of our axioms, including direct modulations consisting of 7 chords between any two major keys. By structure of our axiom system, in the case of a correctable piece with finitely many chords, compliance with classical harmony can be equivocally decided knowing merely the \emph{chord sequence} of the piece, and thus we also see that for these pieces, our axiom system is also \emph{complete}. We note that one of the chord-changing rules, the \emph{principle of least motion}, is hard but possible to formalize mathematically precisely in full generality. There exist interesting mathematical results about this principle in the literature, see e.g. \cite[p.~4--6.]{tymoczko}. 


The logical ordering of musical notions and the mathematically simpler results in this paper can now be used for writing a new classical harmony coursebook. We also plan to do experiments on the possibilities and barriers of composing four-part chorales by Markov models, revising the results of \cite{koralocska}.

\appendix

\section{Modulations} \label{modulations}

In this last extra section, we sketch the classical compositional principles for modulations with seven chords. Here we do not enumerate all exact details of the technically rather complicated compositional principles themselves, neither the \emph{altered chords}, the elements of the convergence areas of the keys not detailed so far. Our whole model for modulations that can be found in a full Hungarian version (not mentioned here for anonymity), in which all the altered chords of the keys and precise formulations of the modulational axioms takes place, is mathematically complete ---but still far from universal, as it only describes modulations consisting of seven chords. 
The seven chords of the modulation do not have to ensure that there is tonality with key $T_1$ in the starting point of the modulations, but the new key $T_2$ has to be established by a complete authentic cadence, according to the compositional principles.

Definition \ref{modka} for modulations guarantees that in the context of modulations it is enough to consider feasible and pauseless homophonic four-part pieces. Pauselessness ensures that chords that cannot follow each other by actual chord-changing also will not occur directly after each other, separated by a pause. In general, pauses can weaken the impact of irregular chord progression and there are some examples in music history when composers use this. But in the case of modulations, in Viennese classicism, the basic aim is to make the key change as smooth as possible and to find some connection between the beginning key and the target key, therefore such trickery is not advised.

In the following, we establish the notions that are necessary to state the remaining compositional principles for modulations. First, we define pauseless extensions of general homophonic four-part pieces, in order to obtain a completely pauseless paradigm for the modulations that incorporates non-pauseless pieces as well. Using these, we define chord sequences, which provides a simpler interpretation for feasible pieces than the one described in Section 4 of our paper. In the same time, note that the chorales cannot be defined mathematically precisely using only chord sequences, therefore also the continuous time construction from our Section 4 is useful.

Let $M$ 
be a homophonic four-part piece that has no packing point apart from its endpoint, then we have a (possibly finite) sequence of disjoint consecutive intervals $(I_i)$ of which $B(M)$ consists, for all of which either $I_i \subseteq \mathrm{Dom}~M$ and $M$ has the constant value of a chord with $\inf I_i$ and $\sup I_i$ being either chord-changing points or boundary points of $\mathrm{Dom}~M$, or $I_i$ is a maximal general pause interval in the sense that $\inf I_i$ and $\sup I_i$ are boundary points of $\mathrm{Dom}~M$. In this case, there is a simple way to construct a pauseless extension $\overline{M}$ of $M$, given by an extension from $\mathrm{Dom}~M$ to $B(M)$, this is called the \emph{right-invariant pauseless extension} of $M$:  \[ \overline{M}(t)= \begin{cases} M(t), ~ \text{if}~ t \in \mathrm{Dom}~M, \\ M[I_i]=M(\sup \lbrace u \in \mathrm{Dom}~M \vert u<t \rbrace), ~\text{if} ~t \in I_{i+1},I_{i+1} \cap \mathrm{Dom}~M=\emptyset. \end{cases} \] \\ In this case, $(\overline{M}(t_i), t_i \in A(\overline{M}))$ is called the \emph{chord sequence} of $M$, here $t_i$'s follow each other in their order in $\mathbb{R}_0^+$. We omit the proof of the following proposition, which shows the role of the playing function group $PL(\mathbb{R})$ in the topology of four-part pieces.

\begin{proposition}
Let $M_1$ and $M_2$ be homophonic four-part pieces such that the values of the chord sequences of $M_1$ and $M_2$ are the same. Then $\exists \theta \in PL(\mathbb{R})$: $M_2=M_1 \circ \theta$.
\end{proposition}

When describing modulations, we will not make any difference between feasible homophonic four-part pieces which have the same chord sequence. The expressions ‘‘a chord sequence is tonal/is correctable/complies with classical harmony'' will be used in the sense that every homophonic four-part piece with the given chord sequence has this property. 

Firstly, we have to establish a connection between modulational axioms and the definition of the correctable piece. We call modulations which satisfy not only Definition \ref{modka} but also the first modulational axiom (Compositional principle \ref{modi}) \emph{basic modulations}. For a $T_1 \to T_2$ modulation that complies with classical harmony apart from chord changes, the next necessary condition that we require is tonality in the starting point of the first degree triad of $T_1$ that opens the modulation, with key $T_1$, and to also tonality in the endpoint of the first degree triad of $T_2$ that opens the modulation (the existence of these chords is guaranteed by the first modulational axiom).

Now we can turn to the basic idea of modulations complying with classical harmony: the chord sequence of the ---pauseless, feasible--- $T_1 \to T_2$ modulation section has to be able to be divided into three disjoint segments (left-closed, right-open intervals) that cover the whole chord sequence \cite[p.~36]{statusquo}:
\begin{description}
\item[Neutral phase \emph{(N)}] In this segment, which is opened by the Ist degree triad, the key of the piece is still $T_1$ (i.e., each element of $N$ is a member of $CA(T_1)$), but there are no secondary dominant chords. In the whole modulation after the first chord there is neither in $T_1$ nor in $T_2$ any root position Ist degree triad until the tonic main triad of $T_2$ occurs and closes the modulation. 
\item[Fundamental step \emph{(F)}] If $T_1$ and $T_2$ are of the same type and they are neighbours in the circle of fifths, this whole segment may be empty. Otherwise here dominant chords of different keys follow each other. Only the last can be a (major/diminished/in the case of minor $T_1$ and minor $T_2$ augmented) triad, the ones before have to be seventh chord inversions. These seventh chords have to follow each other by \emph{elision}\footnote{ We generally use the word ‘‘elision'' for chord progression of inversions of different seventh chords, without the first seventh chord resolving to its tonic. The lack of resolution is expressed by ‘‘elision'', a word of Greek origin for ‘‘omission''. Chord progression using elision always has to use \emph{chromatics} in order to make it possible to comply with classical harmony. Chromatics means the sequence of at least two \emph{semitone steps} after each other in one voice. This semitone sequence has to be such that for each segment of it where the consecutive semitone steps are taken in the same direction, there exists a key in the scale of which, exactly every other step changes degree (i.e, every other step is a minor second and the remaining steps are augmented primes).}. The last chord of $F$ may be a triad and the previous chord may resolve to it.
\item[Cadence \emph{(C)}] The modulation has to be finished by a complete authentic cadence in the new key $T_2$, this shows and stabilizes the tonality in the new key. It may occur that we do not write a cadence in each key, but a modulation progress is only finished when we reach a cadence in some key. It is sure that the last chord before the closing degree I triad is the degree V triad or degree V dominant seventh chord of $T_2$ in the segment $C$. The tonic and subdominant chords of $T_2$ that preceed this chord belong already to $C$ (and not $F$) if and only if they are built up from scale tones in $T_2$, otherwise they belong to $F$.
\end{description}
If a chord sequence of a basic modulation has all the properties that we have introduced in this section, and it is the member of one of the following three modulation types, then we say that it complies with classical harmony apart from its chord-changing points. If its chord-changings are also correct, we say that the modulation complies with classical harmony. The three possible modulation types are:
\begin{description}
\item[Diatonic] For the last chord $H$ of $N$ we have $H \in CA(T_2)$, and every chord after this is convergent in $T_2$. This time $F$ usually consists of at most one chord. This is the smoothest possible key change type, but it is often not possible between keys further away from each other.
\item[Enharmonic] The last chord of $N$ or the first chord of $F$ is an element of $T_1$ that is enharmonic with some element of $CA(T_2)$. The most common enharmonic modulation types use the enharmonic equivalence of diminished seventh chords or augmented triads in different keys. After this chord occurs, we consider it as an element of $CA(T_2)$, and make a chord progression in $T_2$ ending with an authentic cadence.
\item[Chromatic] There is elision in the modulation chord sequence. Very far away leading modulations, such as C major $\to$ F$\sharp$ major can be most conveniently feased this way. In most of the cromatic modulations $\sharp F \geq 2$ holds.
\end{description}
These three categories do not exclude each other pairwise, while it is difficult to accomplish a modulation that is both diatonic and enharmonic at the same time. In the music score collection of the thesis we show examples of both enharmonic and chromatic and both diatonic and chromatic modulations.
Modulational compositional principles finish our axiomatization work. 

\newpage
\section{Convergent chords in major and minor keys} \label{durmoll}
In a major key $T$:

\vspace{5pt}
\hspace{-14pt}
\begin{tiny}
\begin{tabular}{|ll|c|c|c|c|c|}
\hline
Notation & Type & Convergent inversions & Function & Typical following chords \\ \hline 
 1. Diatonic & triads: & & & \\ \hline 
 $\mathrm{I}$ & major triad & all & T & almost all elements of $CA(T)$  \\ \hline
$\mathrm{II}$ & minor triad & root, first & S & $\mathrm{V}^{(7)}$, $\mathrm{II}_{\sharp}^7$ \\ \hline
 $\mathrm{III}$ & minor triad & root, first & - & $\mathrm{IV},~\mathrm{VI},~\mathrm{III}_{\sharp}^7$ \\ \hline
 $\mathrm{IV}$ & major triad & all & S & $\mathrm{V}^{(7)},~ \mathrm{II},~\underset{\sharp}{\mathrm{IV}^{7\flat}},~\mathrm{I}$ \\ \hline
 $\mathrm{V}$ & major triad & root, first & D & $\mathrm{I},~\mathrm{VI},~\underset{\sharp}{\mathrm{III}_5^6},~\underset{\sharp}{\mathrm{V}^7}$ \\ \hline
 $\mathrm{VI}$ & minor triad & root, first & T & $\mathrm{II}$, $\mathrm{IV}$, $\mathrm{VI}_{\sharp}^7$ \\ \hline
 $\mathrm{VII}$ & diminished triad & first (with the third dupl.) & D & $\mathrm{I}^{(6)}$, $\mathrm{VII}^{7\flat}$ \\ \hline
 2. Diatonic & sevenths:  & & & \\ \hline
 $\mathrm{I}^7$ & major minor seventh & all & T & $\mathrm{IV}$ \\ \hline
$\mathrm{II}^7$ & minor major seventh & all & S & $\mathrm{V}$ \\ \hline
$\mathrm{III}^7$ & minor major seventh & all & - & $\mathrm{VI}$ \\ \hline
 $\mathrm{IV}^7$ & major minor seventh & all & - & $\mathrm{VII}^6$, $\mathrm{V}$ \\ \hline
$\mathrm{V}^7$ &   dominant seventh & all & D & $\mathrm{I}$, $\underset{\flat}{\mathrm{I}^2}$ \\ \hline 
 $\mathrm{VI}^7$ & minor major seventh & all & T & $\mathrm{II}$, $\mathrm{VI}_{\sharp}^7$ \\ \hline
 $\mathrm{VII}^7$ & semi-dim. seventh & all & -& $\mathrm{I}$, $\mathrm{III}$ \\ \hline
 3. Diminished & sevenths (altered):  & & & \\ \hline
 $\underset{\sharp}{\mathrm{I}^{7\flat}}$ & dim. seventh & all & - & $\mathrm{II}$, $\mathrm{II}_2^{4(\sharp)}$ \\ \hline
$\underset{\sharp}{\mathrm{II}_{\sharp}^7}$ & dim. seventh & all & - & $\mathrm{III}$, $\mathrm{III}_2^{4(\sharp)}$ \\ \hline 
 $\underset{\sharp}{\mathrm{IV}^{7\flat}}$ & dim. seventh & all & S & $\mathrm{V}$, $\mathrm{V}^2$ \\ \hline
 $\underset{\sharp}{\mathrm{V}^7}$ & dim. seventh & all & - & $\mathrm{VI}$, $\mathrm{II}_2^{4(\sharp)}$ \\ \hline
 $\mathrm{VII}^{7\flat}$ & dim. seventh & all & D & $\mathrm{I}$, $\mathrm{I}^{2(\flat)}$\\ \hline
4. Secondary & dominant sevenths and triads:  & & & \\ \hline
 $\mathrm{I}^{7\flat}$& dominant seventh & all & T & $\mathrm{IV}$, $\underset{\flat}{\mathrm{IV}^2}$ \\ \hline
 $\mathrm{II}^\sharp$& major triad & all & S & $\mathrm{V}$, $\mathrm{V}^2$ \\ \hline
 $\mathrm{II}_{\sharp}^7$& dominant seventh & all & S & $\mathrm{V}$, $\mathrm{V}^2$ \\ \hline
 $\mathrm{III}_{\sharp}^7$ & dominant seventh & all & - & $\mathrm{VI}$, $\mathrm{VI}_2^{4(\sharp)}$ \\ \hline
 $\mathrm{IV}^{7\flat}$ & dominant seventh & third & - & $\mathrm{VII}^6$ \\ \hline
 $\mathrm{VI}_{\sharp}^7$ & dominant seventh & all & - & $\mathrm{II}$, $\mathrm{II}_{2}^{4(\sharp)}$ \\ \hline
 $\mathrm{VII} \tiny{\begin{smallmatrix} 7 \\ 5\sharp \\ \sharp \end{smallmatrix}}$ & dominant seventh & all & - & $\mathrm{III}$, $\mathrm{III}_{2}^{4(\sharp)}$ \\ \hline
 5. Augmented & sixth chords:  & & & \\ \hline
 $\underset{\flat}{\mathrm{VI}}^{6\sharp}$ & ($\sim$dominant seventh) & first (with the third dupl.) & S & $\mathrm{V}$ \\ \hline
$\underset{\flat}{\mathrm{VI}_{5\flat}^{6\sharp}}$ & ($\sim$dominant seventh) & first & S & $\mathrm{V}$ \\ \hline
6. Minor & subdominants \& tonics:  & & & \\ \hline
$\underset{\flat}{\mathrm{II} \tiny{\begin{smallmatrix} 6\sharp \\ 4 \\ 3 \end{smallmatrix}}}$ & irregular & second & S & $\mathrm{V}$ \\ \hline
$\mathrm{II}^{5\flat}$ & diminished triad & first & S & $\mathrm{V}^{(7)}$ \\ \hline
$\mathrm{II}_{5\flat}^7$ & semi-dim. seventh & all & S & $\mathrm{V}^{(7)}$ \\ \hline
$\underset{\flat}{\mathrm{II}}^{6\flat}$ & major triad$\ast$ & first (with the third dupl.) & S & $\mathrm{V}$ \\ \hline
 $\mathrm{VI}^{\flat}$ & minor triad & root, first & S & $\mathrm{V}^{(7)}$ \\ \hline
 $\underset{\flat}{\mathrm{VI}^{5\flat}}$ & major triad$\ast \ast$ & root & - & $\mathrm{V},~\mathrm{II},~\mathrm{IV}$ \\ \hline

\end{tabular}
\end{tiny}
\vspace{10pt}

In total, we have 97 convergent chord inversions. \\
$\ast$ This is the Neapolitan sixth of the minor key with the same first degree as $T$, see below. \\
$\ast \ast$ This is not a subdominant chord, but it comes also from the convergence area of the minor key with the same first degree as $T$. \\
\newpage Note that the augmented sixth chords are formally neither triads nor seventh chords complying with classical harmony according to our definitions, but they are enharmonic to (possibly fifth-deficient) dominant seventh chords. The same applies in minor. Further, the chord $\underset{\flat}{\mathrm{II} \tiny{\begin{smallmatrix} 6\sharp \\ 4 \\ 3 \end{smallmatrix}}}$ does not comply with classical harmony; the same holds for its analogue $\mathrm{II} \tiny{\begin{smallmatrix} 6\sharp \\ 4 \\ 3 \end{smallmatrix}}$ in minor as well as for the dominant nona chord. However, they are widely used in practice, and the axiom system presented in this paper can easily be extended in such a way that these chords also comply with it.  

In a minor key $T$:

\vspace{10pt}
\hspace{-14pt}
\begin{tiny}
\begin{tabular}{|ll|c|c|c|c|c|}
\hline
Notation & Type & Convergent inversions & Function & Typical following chords \\ \hline 
1. Diatonic & triads: & & & \\ \hline 
 $\mathrm{I}$ & minor triad & all & T & almost all elements of $CA(T)$ \\ \hline
 $\mathrm{II}$ & diminished triad & first & S & $\mathrm{V}_{\sharp}^{(7)}$, $\mathrm{II} \tiny{\begin{smallmatrix} 7 \\ 5\sharp \\ \sharp \end{smallmatrix}}$ \\ \hline
 $\mathrm{III}^{5\sharp}$ & augmented triad & root, first & - & $\mathrm{IV},~\mathrm{VI}, \mathrm{I}$ \\ \hline
 $\mathrm{IV}$ & minor triad & all & S & $\mathrm{V}_{\sharp}^{(7)},~ \mathrm{II},~\underset{\sharp}{\mathrm{IV}_{\sharp}^7},~\mathrm{I}$ \\ \hline
$\mathrm{V}^{\sharp}$ & major triad & root, first & D & $\mathrm{I},~\mathrm{VI}$ \\ \hline
$\mathrm{VI}$ & major triad & root, first & - & $\mathrm{II}$, $\mathrm{IV}$, $\mathrm{V}^{\sharp}$ \\ \hline
 $\underset{\sharp}{\mathrm{VII}}$ & diminished triad & first (with the third dupl.) & D & $\mathrm{I}^{(6)}$, $\underset{\sharp}{\mathrm{VII}^{7}}$ \\ \hline
 2. Diatonic & sevenths:  & & & \\ \hline
$\mathrm{I}^{7\sharp}$ & minor augmented seventh B & all & - & $\mathrm{IV}$ \\ \hline
$\mathrm{II}^7$ & semi-diminished seventh & all & S & $\mathrm{V}_{\sharp}^{(7)}$ \\ \hline
 $\mathrm{III}^7$ & augmented major seventh & all &-& $\mathrm{VI},~\mathrm{V}_{\sharp}^7,~\mathrm{I}$ \\ \hline
$\mathrm{IV}^7$ & minor major seventh & all & - & $\mathrm{VII}^6$, $\mathrm{V}_{\sharp}$ \\ \hline
$\mathrm{V}_{\sharp}^7$ & dominant seventh & all & D & $\mathrm{I}$, $\underset{\natural}{\mathrm{I}_2^{4\sharp}}$ \\ \hline 
$\mathrm{VI}^7$ & major minor seventh & all & T & $\mathrm{II}$ \\ \hline
3. Diminished & sevenths (altered):  & & & \\ \hline
 $\underset{\sharp}{\mathrm{VII}^7}$ & diminished seventh & all & D & $\mathrm{I}$, $\mathrm{I}_2^{4\sharp}$ \\ \hline
 $\underset{\sharp}{\mathrm{III}^{7\flat}}$ & diminished seventh & all & - & $\mathrm{IV}$, $\mathrm{IV}_2^{4(\sharp)}$ \\ \hline 
$\underset{\sharp}{\mathrm{IV}_{\sharp}^{7}}$ & diminished seventh & all & S & $\mathrm{V}^{\sharp}$, $\mathrm{V}_2^{4\sharp}$ \\ \hline
4. Secondary & dominant triads and sevenths:  & & & \\ \hline
$\mathrm{I}_{\sharp}^{7}$& dominant seventh & all & T  & $\mathrm{IV}$, $\mathrm{IV}^2$ \\ \hline
$\mathrm{II}_{\sharp}^7$& dominant seventh & all & S & $\mathrm{V}^{\sharp}$, $\mathrm{V}_2^{4\sharp}$ \\ \hline
$\mathrm{III}^{\sharp}$ & major triad & root & - & $\mathrm VI$ \\ \hline
$\underset{\natural}{\mathrm{VII}}$ & major triad & root & - & $\mathrm{III}^7$ \\ \hline
$\underset{\natural}{\mathrm{VII}^7}$ & dominant seventh & all & - & $\mathrm{III}^7$ \\ \hline
5. Augmented & sixth chords:  & & & \\ \hline
 $\mathrm{VI}^{6\sharp}$ & ($\sim$dominant seventh) & first (with the third dupl.) & S & $\mathrm{V}^{\sharp}$ \\ \hline
 $\mathrm{VI}_{5}^{6\sharp}$ & ($\sim$dominant seventh) & first & S & $\mathrm{V}^{\sharp}$ \\ \hline
$\mathrm{II} \tiny{\begin{smallmatrix} 6\sharp \\ 4 \\ 3 \end{smallmatrix}}$ & irregular & second & S & $\mathrm{V}^{\sharp}$ \\ \hline
 6. Picardian & first degree triad:  & & & \\ \hline
 $\mathrm{I}^{\sharp}$ & major triad & root & T & none (final chord) \\ \hline
 7. Neapolitan & sixth chord:  & & & \\ \hline
 $\mathrm{II}^{6\flat}$ & major triad & first (with the third dupl.) & S & $\mathrm{V}^{\sharp}$ \\ \hline
 8. Dominant & nona chord:  & & & \\ \hline
 $\mathrm{V} \tiny{\begin{smallmatrix} 9 \\ 7 \\ \sharp \end{smallmatrix}}$ & dominant nona chord & root (fifth-deficient) & D & $\mathrm{I}$ \\ \hline
\end{tabular}
\end{tiny}
\vspace{10pt}

In total, only 73 convergent chord inversions. 


\subsection*{Acknowledgements} I would like to thank my Bachelor's thesis supervisor \'{A}kos G. Horv\'{a}th, also R. S. Sturman and S. McLaughlin from the University of Leeds (UK) for their advice and support. I would like to thank \'{A}gnes Cseh and M\'{a}t\'{e} V\'{e}csey for proofreading. Theorem \ref{vécsey} was conjectured by  M\'{a}t\'{e} V\'{e}csey.

\addcontentsline{toc}{section}{References}
\bibliographystyle{tMAM}

\end{document}